\documentclass[12pt,a4paper]{amsart}
\usepackage{latexsym}
\usepackage{amsmath,amssymb,amsfonts,amsthm,amscd}
\usepackage[all,arc]{xy}
\usepackage{graphicx}
\usepackage{enumerate}
\usepackage{mathrsfs}
\usepackage{comment}
\usepackage{arydshln}
\usepackage{mathtools}
\usepackage{colonequals}
\usepackage{braket}

\newtheorem{thm}{Theorem}[section]
\newtheorem{cor}[thm]{Corollary}
\newtheorem{prop}[thm]{Proposition}
\newtheorem{lem}[thm]{Lemma}

\theoremstyle{definition}

\newtheorem{ex}[thm]{Example}

\theoremstyle{remark}

\newcommand{\bF}{\mathbb{F}}

\newcommand{\bP}{\mathbb{P}}
\newcommand{\bQ}{\mathbb{Q}}

\newcommand{\cL}{\mathcal{L}}
\newcommand{\cM}{\mathcal{M}}

\newcommand{\cO}{\mathcal{O}}

\newcommand{\cU}{\mathcal{U}}
\newcommand{\cV}{\mathcal{V}}

\newcommand{\sX}{\mathscr{X}}
\newcommand{\sY}{\mathscr{Y}}

\newcommand{\ol}{\overline}
\newcommand{\ul}{\underline}
\newcommand{\wt}{\widetilde}

\newcommand{\cl}{\mathrm{cl}}
\newcommand{\Frob}{\mathrm{Frob}}
\newcommand{\M}{\mathrm{M}}
\newcommand{\id}{\mathrm{id}}
\newcommand{\pr}{\mathrm{pr}}

\DeclareMathOperator{\Spec}{Spec}

\DeclareMathOperator{\Frac}{Frac}
\DeclareMathOperator{\Hom}{Hom}

\DeclareMathOperator{\Stab}{Stab}

\DeclareMathOperator{\Res}{Res}

\DeclareMathOperator{\Ch}{Ch}

\DeclareMathOperator{\SL}{SL}

\DeclareMathOperator{\Spin}{Spin}
\DeclareMathOperator{\SO}{SO}
\DeclareMathOperator{\Cl}{Cl}

\DeclareMathOperator{\Gr}{Gr}
\DeclareMathOperator{\OGr}{OGr}

\makeatletter
\let\c@equation\c@thm
\makeatother
\numberwithin{equation}{section}

\title{Degrees of some orthogonal Deligne--Lusztig varieties}
\author{Yuta Nakayama}
\date{}
\begin{document}

\begin{abstract}
We show a degree formula for a type of orthogonal Deligne--Lusztig varieties and their Pl\"{u}cker embeddings.
This is an analog of work of Li on a unitary case.
\end{abstract}

\maketitle
\setcounter{tocdepth}{1}
\tableofcontents

\section{Introduction}

Deligne--Lusztig varieties are arithmetically important subvarieties of flag varieties over some finite fields.
They originated in \cite{DL} in the context of the geometric construction of representations of finite groups of Lie type.
They also appear in the special fibers of Shimura varieties.

In \cite{LiDeg}, Li proved a degree formula for a kind of unitary Deligne--Lusztig varieties with respect to their Pl\"{u}cker embeddings, where the degree of a projective variety \(X\subseteq \bP^N\) of pure dimension is the geometric intersection number of \(X\) with \(\dim X\) general hyperplanes.
Important in his proof are a case of the Tate conjecture and the identification of special cycles on the Deligne--Lusztig varieties with similar varieties of lower dimension.
The formula is related to \cite[Lemmas 6.4.5, 6.4.6]{LZKR} in the paper toward the Kudla--Rapoport conjecture as a finite field analog of the constant term formula in Kudla's geometric Siegel--Weil formula for unitary Shimura varieties.
Indeed, the Deligne--Lusztig varieties arises in \cite{VWSS} in investigating the supersingular locus of those Shimura varieties.  

In this article, we show a formula similar to that of Li for Deligne--Lusztig varieties of type \({}^2\! D_n\) in \cite{HPGU22}.
The result, especially Proposition \ref{analog} in the course of its proof, would be related to \cite[Lemmas 7.6.5, 7.6.6]{LZSW} in the paper toward the orthogonal variant of the conjecture as before.
In the same paper, the Deligne--Lusztig varieties are used when the authors study the supersingular locus of GSpin Shimura varieties.

We now state our result.
Let \(p\) be an odd prime.
Let \(d\) be a positive integer.
Set \(V\) to be a \(2(d+1)\)-dimensional quadratic space over \(\bF_p\)-vector space.
Assume that its quadratic form is nondegenerate and nonsplit.
Let \(\OGr(d,d+1)\) be the moduli over \(\bF_{p^2}\) of two isotropic subspaces \(L_d\subset L_{d+1}\subset V\otimes \bF_{p^2}\) of dimension \(d\) and \(d+1\).
It has two connected components \(\OGr^\pm(d,d+1)\).
Let \(Y_V^{\pm}\) be the locus inside the components on which
\[
    \dim (L_{d+1}+\Frob^*L_{d+1}) = d+2,
\]
where \(\Frob^*\) is the pull-back by the \(p\)-power Frobenius.
\begin{thm}[Theorem \ref{main}]
The degree of \(Y_V^\pm\) in terms of its Pl\"{u}cker embedding through \(L_{d+1}\) is
\[
    2^d\prod_{i=1}^d \frac{-p^i+1}{-p+1}.
\]
\end{thm}

Our proof is direct like \cite{LiDeg} in that it does not require arguments over local fields in \cite{LZSW}.
We rely on a case of the Tate conjecture and the inductive structure of the special cycles and Deligne--Lusztig varieties that we consider.
Toward the proof, we also develop a moduli description of the Deligne--Lusztig varieties that the author did not find in the literature.
The depiction is necessary for calculating the normal bundle of special cycles.
In turn, we use the bundle in the excess intersection formula.

Finally, we present the organization of this paper.
In \S \ref{sec:Vars}, we describe geometric objects in play.
These include the Deligne--Lusztig varieties, their special cycles, and the line bundles on the varieties such as tangent and normal bundles.
The aforementioned moduli description of \(Y_V^+\amalg Y_V^-\) is in this section.
In \S \ref{sec:Chern}, we perform calculations relevant to Chern classes.
Our main result is here as well.

\subsection*{Acknowledgement}
The author thanks his adviser N.~Imai for academically supporting him throughout this project.
He also thanks C.~Li for a conversation regarding this work.

\subsection{Notations}
For a finite set \(T\), we write \(\sharp T\) for the number of its elements.
The notation \(\M_n(A)\) for a ring \(A\) and a positive integer \(n\) is to be understood as the algebra of \(n\times n\) matrices over \(A\).
If \(W\) is a subspace of some ambient quadratic space, then the orthogonal complement of \(W\) is denoted by \(W^\bot\).
If \(\cV\) is a vector bundle over some scheme, then \(V^\vee\) denotes its dual bundle.
Also, \(\Gr(r, \cV)\) denotes the Grassmanninan parametrizing locally direct summands of \(\cV\) of rank \(r\).
Moduli functors are denoted by their representing objects when the latter exists.

If \(X\to S\) is a smooth morphism of schemes, then \(T_{X/S}\) means its tangent bundle.
Here and in other places we often drop ``\(\Spec\).''
For a regular closed immersion \(i\) of schemes, let \(N_{i}\) be its normal bundle.

The notations below are from \cite[\S 1.4]{LiDeg}.
Let \(X\) be a smooth projective variety over a finite field \(k\).
The notation \(\Ch^r(X_{\ol k})\) stands for the Chow group of codimension-\(r\) cycles on \(X_{\ol k}\).
If \(l\) is a prime nonzero in \(k\), then \(\cl_r\colon \Ch^r(X_{\ol k})_\bQ \to H^{2r}(X_{\ol k},\bQ_l)(r)\) denotes the \(l\)-adic cycle class map.
The class of a subvariety \(Z\subseteq X\) of codimension \(r\) in \(\Ch^r(X_{\ol k})\) is denoted by \([Z]_X\).
Sometimes we just write \([Z]\).
If \(\cV\) is a vector bundle on \(X\), then \(c_r(\cV)\in \Ch^r(X_{\ol k})\) is its \(r\)-th Chern class.
These can mean their images in \(\Ch^r(X_{\ol k})_\bQ\) as well.

\section{Varieties and vector bundles} \label{sec:Vars}

We borrow settings from \cite[\S 3.2]{HPGU22}.
Some notations are changed to those from \cite[\S 4.7]{LZSW}.
Let \(p>2\) be a prime number.
Take a nonnegative integer \(d\).
Set \(V\) to be an \(\bF_p\)-vector space of dimension \(2(d+1)\).
Take a nondegenerate nonsplit symmetric form \([*,*]\colon V\times V\to \bF_p\).
Fix a basis \(e_1,\dots, e_{d+1}, f_1,\dots, f_{d+1}\) of \(V\otimes \bF_{p^2}\) such that \([e_i,f_j] = \delta_{i,j}\) for \(1\leqq i,j\leqq d+1\) and Kronecker delta, that \(e_1,\dots,e_d,f_1,\dots,f_d\in V\), and that the \(p\)-power Frobenius switches \(e_{d+1}\) and \(f_{d+1}\).

\subsection{Deligne--Lusztig varieties}
For \(0\leqq r \leqq d+1\), let \(\OGr(r)\) be the functor from the category of \(\bF_{p^2}\)-schemes to that of sets carrying \(S\) to the set of totally isotropic elements of \(\Gr(r, V\otimes \bF_{p^2})(S)\).
The functor is represented by a smooth \(\bF_{p^2}\)-scheme denoted by the same symbol as in the case of Grassmannians.
Similarly, let \(\OGr(d,d+1)\) be the smooth \(\bF_{p^2}\)-scheme that represents the functor with the same source and target carrying \(S\) to the set of \(\cL_d\subset \cL_{d+1}\) where \(\cL_r\in\OGr(r)(S)\) for \(d\leqq r\leqq d+1\).
By \cite[Lemma 3.4]{HPGU22}, the forgetful functor allows us to see \(\OGr(d,d+1)\) as a trivial double cover \(\OGr^+(d,d+1)\amalg\OGr^-(d,d+1)\) over \(\OGr(d)\).
We decide that \((\langle e_1,\dots,e_d\rangle, \langle e_1,\dots, e_{d+1}\rangle )\in \OGr^+(d,d+1)(\bF_{p^2})\).

Let \(\sY_V\) be the \(\bF_{p^2}\)-scheme representing the functor from the category of \(\bF_{p^2}\)-schemes to that of sets carrying \(S\) to the set of \((\cL_d, \cL_{d+1})\in\OGr(d,d+1)(S)\) such that \(\cL_d\subset \Frob^*\cL_{d+1}\) and that
\[
    \cL_{d+1}+\Frob^*\cL_{d+1}\subseteq V\otimes \cO_S
\]
is a locally direct summand, where \(\Frob\colon S\to S\) is the \(p\)-power Frobenius.
Note \(\cL_{d+1}\neq \Frob^*\cL_{d+1}\) at each point of \(S\) by the nonsplitness condition.

\begin{lem} \label{sm}
The \(\bF_{p^2}\)-scheme \(\sY_V\) is smooth.
\end{lem}

\begin{proof}
As below, \(\sY_V\) is formally smooth.
Let \(S_0\to S\) be a closed immersion of affine \(\bF_{p^2}\)-schemes whose ideal sheaf vanishes when squared.
Take \((\cL_{0,d}, \cL_{0,d+1})\in \sY_V(S_0)\).
Since \(\OGr(d+1)\) is smooth, some \(\cL_{d+1}\in \OGr(d+1)(S)\) lifts \(\cL_{0,d+1}\).
Note that \(\Frob^*\cL_{d+1}\) only depends on \(\cL_{0, d+1}\) because the \(p\)-th power of the ideal sheaf of the closed immersion \(S_0\to S\) is zero.
Since \(\Gr(d, \Frob^*\cL_{d+1})\) is smooth, some \(\cL'_d\in\Gr(d, \Frob^*\cL_{d+1})(S)\) lifts \(\cL_{0,d}\).
This \(\cL'_d\) is totally isotropic.
By \cite[Lemma 3.4]{HPGU22}, there is a unique \(\cL'_{d+1}\in\OGr(d+1)(S)\) such that \((\cL'_d, \cL'_{d+1})\in \OGr(d,d+1)(S)\) extends \((\cL_{0,d}, \cL_{0,d+1})\).

It is enough to prove that \(\cL'_{d+1}+\Frob^*\cL'_{d+1}\) is a locally direct summand of \(V\otimes \cO_S\).
By shrinking \(S\), we may assume that there is a retract \(V\otimes \cO_{S_0}\to \cL_{0,d+1}+\Frob^*\cL_{0,d+1}\).
Lift this to \(V\otimes \cO_{S}\to \cL'_{d+1}+\Frob^*\cL'_{d+1}\).
Its composition
\[
\cL'_{d+1}+\Frob^*\cL'_{d+1}\to V\otimes \cO_S\to\cL'_{d+1}+\Frob^*\cL'_{d+1}
\]
with the inclusion is surjective by the lemma of Nakayama.
We move on to the injectivity.
The orthogonal complement of \(\cL_{0,d}\) is \(\cL_{0,d+1}+\Frob^*\cL_{0,d+1}\) since both are locally direct summands of \(V\otimes \cO_S\) of the same rank.
The same is true for \(\cL'_d\) and \(\cL'_{d+1}\) by the lemma of Nakayama.
In particular, \(\cL'_{d+1}+\Frob^*\cL'_{d+1}\) is a vector bundle.
The required injectivity follows by applying the same lemma to the kernel.
\end{proof}

\begin{lem} \label{cl}
The forgetful map \(\sY_V\to \OGr(d+1)\) is a closed immersion.
\end{lem}
\begin{proof}
First we show it is mono.
For an \(\bF_{p^2}\)-scheme \(S\) and \((\cL_d,\cL_{d+1})\in \sY_V(S)\), we have \(\cL_d = \cL_{d+1}\cap \Frob^*\cL_{d+1}\) because their orthogonal complements are locally direct summands of \(V\otimes\cO_S\) of the same rank.
Now we turn to the properness of \(\sY_V\to \OGr(d+1)\).
We argue by the valuative criterion.
Let \(R\) be a discrete valuation ring over \(\bF_{p^2}\).
Take \(L_{d+1}\in \OGr(d+1)(R)\) and \((L_{0,d}, L_{d+1}\otimes_R \Frac R)\in \sY_V(\Frac R)\).
Put \(L_d\colonequals L_{0,d}\cap V\otimes R\subset V\otimes\Frac R\).
This is a free isotropic \(R\)-module of rank \(d\).
We check \((L_d,L_{d+1})\in \sY_V(R)\).
The condition \(L_d\subset L_{d+1}\otimes_{R,\Frob}R\) is clear from the one for \((L_{0,d}, L_{d+1}\otimes_R \Frac R)\).
The remaining condition is also straightforward, e.g., as below.
For a splitting \((V\otimes R)/(L_{d+1}\otimes_{R,\Frob} R)\xrightarrow\sim M\subset V\otimes R\), we have
\(
    V\otimes R/(L_{d+1}+L_{d+1}\otimes_{R,\Frob}R)\simeq M/(M\cap L_{d+1}).
\)
The right side is torsion-free since \(M\cap L_{d+1} = (M\cap L_{d+1})\otimes_R \Frac R\cap V\otimes R\subset V\otimes \Frac R\).
\end{proof}

Our \(\sY_V\) coincides with \(\sX\) in \cite[\S 3.2]{HPGU22} by Lemmas \ref{sm} and \ref{cl}.

Let \(Y_V^\pm\subset \sY_V\) be the pull-back of \(\OGr^\pm(d,d+1)\) by the forgetful morphism.
These are indeed projective geomtrically irreducible Deligne--Lusztig varieties by \cite[Proposition 3.6]{HPGU22}.
We fix one of the \(+\) and \(-\) irrespective of \(V\) and denote the corresponding pull-back \(Y_V\).

\subsection{Special cycles}

For \(0\leqq r\leqq d\), let \(Y_{V,r}\) be the reduced closed subscheme of \(Y_V\) such that
\[
    Y_{V,r}(k) = \{(L_d,L_{d+1})\in Y_V(k)\mid L_{d+1}^{(r+1)} = L_{d+1}^{(r+2)}\}
\]
for a field \(k\supseteq \bF_{p^2}\), where \(L_{d+1}^{(i)}\colonequals L_{d+1}\cap \Frob^*L_{d+1}\cap \dots\cap(\Frob^{i})^*L_{d+1}\) for a nonnegative integer \(i\).
We also put \(Y_{V,-1}\colonequals\emptyset\).
A description of \(Y_{V,r}\) via Deligne--Lusztig subvarieties of \(Y_V\) is in \cite[Proposition 3.8]{HPGU22}.
If \(W\subset V\) is an isotropic subspace, then we have a morphism \(j_{W,V}\colon Y_{W^\bot/W}\to Y_V\) that carries an element of \(Y_{W^\bot/W}(S)\) for an \(\bF_{p^2}\)-scheme \(S\) to its inverse image by the map \(W^\bot\otimes \cO_S\to (W^\bot /W)\otimes \cO_S\).

In fact, the following holds similarly to \cite[Proposition 5.7]{RTWSSUnitRam}.

\begin{prop}
For \(0\leqq r\leqq d\), there is a unique bijection between the set of irreducible components of \(Y_{V,r}\backslash Y_{V,r-1}\) and the \((d-r)\)-dimensional isotropic subspaces of \(V\), such that the closure in \(Y_V\) of the component corresponding to \(W\subset V\) is isomorphic to \(Y_{W^\bot/W}\) through the above morphism.
\end{prop}

This leads to the following collorary of \cite[Theorem 6.3.2]{LZSW}.
For a prime \(l\neq p\) and \(0\leqq r\leqq d\), we write \(T_l^{2r}(V)\) for the subspace of \(H^{2r}(Y_{V,\ol{\bF_p}}, \bQ_l)(r)\) generated by \(\cl_r(Y_{W^\bot/W})\) for \((d-r)\)-dimensional isotropic subspaces \(W\subset V\).
\begin{prop} \label{Poincare}
The Poincar\'e duality induces a perfect pairing
\[
    T_l^{2r}(V)\times T_l^{2(d-r)}(V)\to \bQ_l.
\]
\end{prop}

\subsection{Vector bundles}

Here we focus on \(N_{j_{W,V}}\).
It is imperative, appearing in the excess intersection formula.
By \((\cU_d, \cU_{d+1})\), we mean the universal object over \(Y_V\).
Put \(\cL_{Y_V}\colonequals\cU_{d+1}/\cU_d\).
Note \(j_{W,V}^*\cL_{Y_V}\simeq \cL_{Y_{W^\bot/W}}\) for a totally isotropic subspace \(W\subset V\).
\begin{prop} \label{tang}
We have
\[
    T_{Y_V/\bF_{p^2}}\simeq \ul\Hom(\cL_{Y_V}, V\otimes \cO_{Y_V}/(\cU_{d+1}+\Frob^*\cU_{d+1})).
\]
\end{prop}

\begin{proof}
The left side is the Zariski sheaf by which an open subset \(S\subseteq Y_V\) is carried to the inverse image by \(Y_V(\Spec \cO_S[\epsilon]/(\epsilon^2))\to Y_V(S)\) of the open immersion.
We know an isomorphism between \(T_{\OGr(d+1)/\bF_{p^2}}\) pulled back to \(Y_V\) and \(\wedge^2(V\otimes\cO_S/\cU_{d+1})\), namely the subsheaf of \(\ul\Hom(\cU_{d+1}, V\otimes\cO_S/\cU_{d+1})\) made of sections \(l\) such that
\begin{equation}
    [l(u),u']+[u,l(u')] = 0, \label{alt}
\end{equation}
where \(u\) and \(u'\) are sections of \(\cU_{d+1}\) and the pairings are well-defined because \(\cU_{d+1}\) is isotropic.
It is enough to compare sections of the left and right side of the proposition on sufficiently small affine \(S\subseteq Y_V\) via the last isomorphism.

Put
\begin{align*}
    \cL_d&\colonequals \cU_d|_S, \\
    \cL_{d+1}&\colonequals \cU_{d+1}|_S, \\
    \cL&\colonequals \cL_{Y_V}|_S, \\
    \cM&\colonequals \Frob^*\cL_{d+1}/\cL_d, \\
    \cM_d&\colonequals V\otimes \cO_S/(\cL_{d+1}+\Frob^*\cL_{d+1}).
\end{align*}
We may assume that these are all free.
We take splittings \(\cL\to \cL_{d+1}\), \(\cM\to \Frob^*\cL_{d+1}\) and \(\cM_d\to V\otimes \cO_S\).
Put \(\cM_{d+1}\colonequals\cM_d\oplus\cM\).

Let \((\wt\cL_d,\wt\cL_{d+1})\in T_{Y_V/\bF_{p^2}}(S)\).
This is determined by some \(l\colon \cL_{d+1}\to \cM_{d+1}\).
We have \(\Frob^*\wt\cL_{d+1} = \Frob^*\cL_{d+1}\otimes_{\cO_S}\cO_S[\epsilon]/(\epsilon^2)\) because \(\epsilon^p=0\).
Thus \(\Frob^*\wt\cL_{d+1}\) includes both \(\wt\cL_d\) and \(\cL_{d}\otimes_{\cO_S} \cO_S[\epsilon]/(\epsilon^2)\).
We also have \(V\otimes \epsilon\cO_S\cap \Frob^*\wt\cL_{d+1} = \epsilon\Frob^*\cL_{d+1} \subset V\otimes \cO_S[\epsilon]/(\epsilon^2)\).
Recall that \(\wt\cL_d\) is generated by the image of \(\id_{\cL_d} + \epsilon l|_{\cL_d}\).
The last three sentences imply that \(l(\cL_d)\subseteq\cM\).
By this, the perfectness of \([*,*]\) on \(\cL\times \cM\), and (\ref{alt}) for sections of \(\cL_d\) and \(\cL\),
\[
    l|_\cL\colon \cL\to \cM_{d+1}\cap\cL^\bot\simeq \cL_d^\vee\simeq V\otimes \cO_S/(\cL_{d+1}+\Frob^*\cL_{d+1})
\]
knows \(l|_{\cL_d}\) as well.

This means that the left side of the proposition is in the right side through the similar isomorphism for \(\OGr(d+1)\).
We get the proposition by comparing the rank and noting that the both sides are locally direct summands of the both sides of the known isomorphism.
\end{proof}

\begin{cor} \label{normal}
We have \(N_{j_{W,V}}\simeq \ul\Hom(\cL_{Y_{W^\bot/W}}, (V/W^\bot)\otimes\cO_{Y_{W^\bot/W}})\).
\end{cor}
\begin{proof}
Note the short exact sequence
\[
    0\to T_{Y_{W^\bot/W}/\bF_{p^2}}\to j_{W,V}^*T_{Y_V/\bF_{p^2}}\to N_{j_{W,V}}\to 0.
\]
The corollary follows from the compatible isomorphisms in Proposition \ref{tang} for the left and middle terms of the sequence together with the isomorphism \(j_{W,V}^*\cL_{Y_V}\simeq \cL_{Y_{W^\bot/W}}\).
\end{proof}

\begin{ex} \label{1D}
We calculate \(\cL_{Y_V}\) by hand in the case where \(d = 1\).
Fix \(a\in \bF_p\) that is not a square of an element in \(\bF_p\).
Also fix \(b\in \bF_{p^2}\) satisfying \(b^2 = a\).
There exists a basis \(X_1,\dots, X_4\) of \(V\) such that
\[
    [x_1X_1+\dots+x_4X_4, x_1X_1+\dots+x_4X_4] = x_1^2-x_2^2+ax_3^2-x_4^2
\]
for \(x_1,\dots, x_4\in\bF_p\) and that
\begin{align*}
    e_1 &= X_1+X_2\in V, \\
    f_1 &= \frac{X_1 - X_2}{2}\in V, \\
    e_2 &= X_3\otimes b^{-1} + X_4\otimes 1\in V\otimes \bF_{p^2}, \\
    f_2 &= \frac{X_3\otimes b^{-1} - X_4\otimes 1}{2}\in V\otimes \bF_{p^2}.
\end{align*}

We consider \(\Spin(V)\).
Let \(\Cl^0(V)\) be the even Clifford algebra of \(V\).
Its center is \(\bF_p\oplus \bF_pX_1\dotsm X_4\), isomorphic to \(\bF_{p^2}\) by carrying \(X_1\dotsm X_4\) to \(b\).
Moreover, \(\Cl^0(V)\) is isomorphic to the algebra \(M_2(\bF_{p^2})\) by
\begin{alignat*}{2}
    X_1X_2&\mapsto \begin{pmatrix} 
        1 & 0 \\
        0 & -1
    \end{pmatrix},\hspace{1mm}
    &X_1X_3\mapsto \begin{pmatrix} 
        0 & -b \\
        b & 0
    \end{pmatrix},\hspace{1mm}
    &X_1X_4\mapsto \begin{pmatrix} 
        0 & 1 \\
        1 & 0
    \end{pmatrix}, \\
    X_2X_3&\mapsto \begin{pmatrix} 
        0 & b \\
        b & 0
    \end{pmatrix},\hspace{1mm}
    &X_2X_4\mapsto \begin{pmatrix} 
        0 & -1 \\
        1 & 0
    \end{pmatrix},\hspace{1mm}
    &X_3X_4\mapsto \begin{pmatrix} 
        b & 0 \\
        0 & -b
    \end{pmatrix}.
\end{alignat*}
We identify both sides of the isomorphism.
Let \(\Gamma^0\subset \Cl^0(V)\) be the even Clifford group.
We also see it as an algebraic group over \(\bF_p\).
The spinor norm of a section of \(\Gamma^0\) considered over some algebra is given by its determinant.
Thus \(\Spin (V) = \Res_{\bF_{p^2}/\bF_p}\SL_2\), noting the exceptional isomorphism \(D_2\simeq A_1\times A_1\) of Dynkin diagrams.

We embed \(V\) into \(\M_2(\bF_{p^2})\) by
\begin{alignat*}{4}
    &X_1&\mapsto& \begin{pmatrix} 
        0 & -1 \\
        1 & 0
    \end{pmatrix},\hspace{1mm}
    &X_2&\mapsto& \begin{pmatrix} 
        0 & -1 \\
        -1 & 0
    \end{pmatrix}, \\
    &X_3&\mapsto& \begin{pmatrix} 
        -b & 0 \\
        0 & -b
    \end{pmatrix},\hspace{1mm}
    &X_4&\mapsto& \begin{pmatrix} 
        1 & 0 \\
        0 & -1
    \end{pmatrix}.
\end{alignat*}
The last \(\M_2(\bF_{p^2})\) is unrelated to \(\Cl^0(V)\).
By the embedding, the action of \(\Gamma^0\) on \(V\) translates into \(\gamma v = \Frob({}^t\!\gamma)^{-1}v{}^t\!\gamma\).
This means that \(\Spin (V)_{\bF_{p^2}}\simeq \SL_2^2\) acts on \(V\otimes\bF_{p^2}\simeq \M_2\) by \((\alpha, \beta) v = {}^t\!\beta^{-1}v\alpha\).

About \(\SO(V)_{\bF_{p^2}}\), recall from \cite[\S 3.2]{HPGU22} the Borel and maximal parabolic subgroups
\begin{alignat*}{2}
    &B &=& \Stab_{\SO(V)_{\bF_{p^2}}}(\bF_{p^2}e_1) \\
    \subset&P^+ &=& \Stab_{\SO(V)_{\bF_{p^2}}}(\bF_{p^2}e_1\oplus\bF_{p^2}e_2).
\end{alignat*}
Let \(B'\) and \(P'^+\) be the inverse image by \(\Spin(V)_{\bF_{p^2}}\to \SO(V)_{\bF_{p^2}}\) of \(B\) and \(P^+\), respectively.
The former is the product of two copies of the group of lower triangular matrices with determinant \(1\).
The latter is the product of the same group and \(\SL_2\).
We similarly recall and define \(P^-\) and \(P'^-\), replacing \(e_2\) with \(f_2\).
Our \(P'^-\) switches the order of direct product in \(P'^+\).
Because of the proof of \cite[Proposition 3.6]{HPGU22}, in our case \(Y_V^+\) is the graph of \(\bP^1\simeq\Spin(V)_{\bF_{p^2}}/P'^+\to \Spin(V)_{\bF_{p^2}}/P'^-\simeq\bP^1\), where the arrow is the relative Frobenius.
Furthermore, \(\cL_{Y_V^+}\) is pulled back from \(\OGr^+(1,2)\simeq \Spin(V)_{\bF_{p^2}}/B'\simeq (\bP^1)^2\).
We calculate the bundle on \(\Spin(V)_{\bF_{p^2}}/B'\).
A section
\[
    \left(\begin{pmatrix}
        p&0 \\
        q&p^{-1}
    \end{pmatrix},
    \begin{pmatrix}
        p'&0 \\
        q'&p'^{-1}        
    \end{pmatrix}\right)
\]
of \(B'\) considered over some algebra acts on the fiber \(\langle e_1,e_2\rangle/\langle e_1\rangle\) of the bundle over the image of unit by carrying \(e_2\) to
\[
    \prescript{t\!}{}{
    \begin{pmatrix}
        p'&0 \\
        q'&p'^{-1}
    \end{pmatrix}^{-1}}e_2\prescript{t\!}{}{\begin{pmatrix}
        p&0 \\
        q&p^{-1}        
    \end{pmatrix}}=\begin{pmatrix}
        0&2p^{-1}q' \\
        0&-2p^{-1}p'
    \end{pmatrix}
    =p^{-1}p'e_2
\]
in the same fiber modulo \(e_1\).
Thus the line bundle is \(\pr_1^*\cO_{\bP^1}(-1)\otimes\pr_2^*\cO_{\bP^1}(1)\), where \(\pr_i\) is the \(i\)-th projection for \(i = 1, 2\).
This pulls back to \(\cL_{Y_V^+} = \cO_{\bP^1}(p-1)\).
We also have \(\cL_{Y_V^-} = \cO_{\bP^1}(p-1)\), substituting \(\langle e_1,f_2\rangle\) for \(\langle e_1, e_2\rangle\).
\end{ex}

\section{Chern classes} \label{sec:Chern}

Here we discuss the degree formula and the finite field version of the constant term formula in Kudla's geometric Siegel--Weil formula for GSpin Shimura varieties.
The argument depends on the inductive structure of special cycles.
We assume \(d>0\).
We identify the top rational Chow group of \(Y_{V,\ol{\bF_p}}\) and of its special cycles with \(\bQ\).

\begin{lem} \label{1Dcyc}
    Let \(W\subset V\) be a \((d-1)\)-dimensional isotropic subspace.
    Then \(c_1(\cL_{Y_V}^\vee)[Y_{W^\bot/W}] = -p+1\) in \(\Ch^d(Y_{V,\ol{\bF_p}})_{\bQ}\).
\end{lem}
\begin{proof}
By the projection formula and Example \ref{1D},
\[
    c_1(\cL_{Y_V}^\vee)[Y_{W^\bot/W}] = j_{W,V,*}(c_1(j_{W,V}^*\cL_{Y_V}^\vee)) = j_{W,V,*}(c_1(\cL_{Y_{W^\bot/W}}^\vee)) = -p+1.
\]
\end{proof}

\begin{prop} \label{induct}
    We have in \(\Ch^1(Y_{V,\ol{\bF_p}})_\bQ\)
    \[
        c_1(\cL_{Y_V}^\vee) = \frac{-p+1}{p^{d+1}+1}\sum_{W\in\bP(V) \hspace{1mm}\mathrm{isotropic}} [Y_{W^\bot/W}].
    \]
\end{prop}
\begin{proof}
    When \(d = 1\), the proposition is the consequence of Example \ref{1D} and \cite[Corollary 3.2.3]{LZSW}, by which the number of isotropic lines in \(V\) is \(p^2+1\).

    We come to the case where \(d>1\).
    By Proposition \ref{Poincare}, it is enough to show that if \(W'\subset V\) is isometric of dimension \(d-1\), then
    \begin{equation}
        c_1(\cL_{Y_V}^\vee)[Y_{W'^\bot/W'}] = \frac{-p+1}{p^{d+1}+1}\sum_{W\in\bP(V)\hspace{1mm}\mathrm{isotropic}} [Y_{W^\bot/W}][Y_{W'^\bot/W'}]. \label{prdw/cyc}
    \end{equation}
    We have three cases about each term in the right side.
    \begin{enumerate}
        \item\label{excess} First, we deal with the case where \(W\subseteq W'\).
            By the combination of the excess intersection formula, Corollary \ref{normal} and Lemma \ref{1Dcyc},
            \begin{align*}
                [Y_{W^\bot/W}][Y_{W'^\bot/W'}] &= j_{W,V,*}(c_1(N_{j_{W,V}})[Y_{W'^\bot/W'}]_{Y_{W^\bot/W}}) \\
                &= j_{W,V,*}(c_1(\cL_{Y_{W^\bot/W}}^\vee)[Y_{W'^\bot/W'}]_{Y_{W^\bot/W}})= -p+1.
            \end{align*}
        \item\label{just} Second, we handle the case where \(W\oplus W'\subset V\) is isotropic.
            Then \(Y_{W^\bot/W}\times_{Y_V} Y_{W'^\bot/W'}=Y_{W^\bot\cap W'^\bot/W\oplus W'}\simeq \Spec \bF_{p^2}\).
            Thus \([Y_{W^\bot/W}][Y_{W'^\bot/W'}] = 1\).
        \item\label{empty} Finally, here is the case where \(W\oplus W'\subset V\) is not isotropic.
            Then \([Y_{W^\bot/W}][Y_{W'^\bot/W'}] = 0\) as the related intersection is empty.
    \end{enumerate}
    The number of \(W\) as in (\ref{excess}) is \(\sharp \bP(W') = (p^{d-1}-1)/(p-1)\).
    The similar number for (\ref{just}) is the number of \(W\in\bP(W'^\bot/W')\) that is isotropic times \(\sharp W'\). This equals \((p^2+1)p^{d-1}\) by \cite[Corollary 3.2.3]{LZSW}.
    Thus the right side of (\ref{prdw/cyc}) is
    \[
        \frac{-p+1}{p^{d+1}+1}\left(\frac{p^{d-1}-1}{p-1}(-p+1)+(p^2+1)p^{d-1}\cdot 1\right) = -p+1.
    \]
    This is the left side by Lemma \ref{1Dcyc}.
\end{proof}

\begin{prop} \label{analog}
We have in \(\Ch^d(Y_{V,\ol{\bF_p}})_{\bQ}\)
\[
    c_1(\cL_{Y_V}^\vee)^d = \prod_{i = 1}^d (-p^i+1).
\]
\end{prop}
\begin{proof}
    We induct on \(d\).
    When \(d = 1\), then it follows from Example \ref{1D}.

    When \(d >1\), we have
    \begin{align*}
        c_1(\cL_{Y_V}^\vee)^d &= \frac{-p+1}{p^{d+1}+1}c_1(\cL_{Y_V}^\vee)^{d-1}\sum_{W\in\bP(V) \hspace{1mm}\mathrm{isotropic}} [Y_{W^\bot/W}] \\
        &= \frac{-p+1}{p^{d+1}+1} \sum_{W\in\bP(V) \hspace{1mm}\mathrm{isotropic}} j_{W,V,*}(c_1(\cL_{Y_{W^\bot/W}}^\vee)^{d-1}) \\
        &= \frac{-p+1}{p^{d+1}+1}\frac{(p^{d+1}+1)(p^d-1)}{p-1}\prod_{i = 1}^{d-1} (-p^i+1) = \prod_{i = 1}^d (-p^i+1),
    \end{align*}
    where the first equality is by Proposition \ref{induct}, the second by the projection formula, and the third by \cite[Corollary 3.2.3]{LZSW} and the induction hypothesis.
\end{proof}

\begin{thm} \label{main}
The degree of \(Y_V\) in terms of its Pl\"{u}cker embedding through \(\cU_{d+1}\) is
\[
    2^d\prod_{i=1}^d \frac{-p^i+1}{-p+1}.
\]
\end{thm}
\begin{proof}
As in the proof of Lemma \ref{sm}, we have \(\cU_d^\bot = \cU_{d+1}+ \Frob^*\cU_{d+1}\).
This together with \(\cU_d^\bot \simeq (V\otimes \cO_{Y_V}/\cU_d)^\vee\) leads to the following:
\begin{align*}
    0 = &c_1(\cU_{d}^\bot)-c_1(\cU_{d}) =c_1(\cU_{d+1} + \Frob^*\cU_{d+1})-c_1(\cU_{d}) \\
    =&c_1(\Frob^*\cU_{d+1}/\cU_d)+c_1(\cU_{d+1})-c_1(\cU_d) =(p+1)c_1(\cU_{d+1})-2c_1(\cU_d).
\end{align*}
This means that
\[
    c_1(\cL_{Y_V}) = c_1(\cU_{d+1}) - c_1(\cU_d) = \frac{-p+1}{2}c_1(\cU_{d+1}).
\]
Through the Pl\"{u}cker embedding into \(\bP^N\), by the projection formula and Proposition \ref{analog},
\begin{align*}
    c_1(\cO_{\bP^N}(1))^d[Y_V] &= c_1(\cO_{\bP^N}(1)|_{Y_V})^d = c_1(\det \cU_{d+1}^\vee)^d = c_1(\cU_{d+1}^\vee)^d \\
    &= 2^d\prod_{i=1}^d \frac{-p^i+1}{-p+1}.
\end{align*}
\end{proof}
\bibliographystyle{amsplain}
%\bibliography{bible}

\begin{thebibliography}{1}

\bibitem{DL}
P.~Deligne and G.~Lusztig, \emph{Representations of reductive groups over
  finite fields}, Ann. of Math. (2) \textbf{103} (1976), no.~1, 103--161.
  \MR{393266}

\bibitem{HPGU22}
B.~Howard and G.~Pappas, \emph{On the supersingular locus of the {${\rm
  GU}(2,2)$} {S}himura variety}, Algebra Number Theory \textbf{8} (2014),
  no.~7, 1659--1699. \MR{3272278}

\bibitem{LiDeg}
C.~Li, \emph{Degrees of unitary {D}eligne--{L}usztig varieties},
  arXiv:2301.08886v1 (2023).

\bibitem{LZKR}
C.~Li and W.~Zhang, \emph{Kudla-{R}apoport cycles and derivatives of local
  densities}, J. Amer. Math. Soc. \textbf{35} (2022), no.~3, 705--797.
  \MR{4433078}

\bibitem{LZSW}
\bysame, \emph{On the arithmetic {S}iegel-{W}eil formula for {GS}pin {S}himura
  varieties}, Invent. Math. \textbf{228} (2022), no.~3, 1353--1460.
  \MR{4419634}

\bibitem{RTWSSUnitRam}
M.~Rapoport, U.~Terstiege, and S.~Wilson, \emph{The supersingular locus of the
  {S}himura variety for {${\rm GU}(1,n-1)$} over a ramified prime}, Math. Z.
  \textbf{276} (2014), no.~3-4, 1165--1188. \MR{3175176}

\bibitem{VWSS}
I.~Vollaard and T.~Wedhorn, \emph{The supersingular locus of the {S}himura
  variety of {${\rm GU}(1,n-1)$} {II}}, Invent. Math. \textbf{184} (2011),
  no.~3, 591--627. \MR{2800696}

\end{thebibliography}

\noindent
Yuta Nakayama\\
Graduate School of Mathematical Sciences, The University of Tokyo,
3-8-1 Komaba, Meguro-ku, Tokyo, 153-8914, Japan \\
nkym@ms.u-tokyo.ac.jp
\end{document}